\documentclass[10pt]{amsart}
\usepackage{mathtools}
\usepackage{amsmath}
\usepackage{amssymb}
\usepackage[utf8]{inputenc}
\usepackage{amsthm}
\usepackage{bbm}
\usepackage{hyperref}
\hypersetup{
	colorlinks=true,
	linkcolor=blue,
	filecolor=magenta,      
	urlcolor=cyan,
	pdftitle={Overleaf Example},
	pdfpagemode=FullScreen,
}

\newtheorem{theorem}{Theorem}
\newtheorem{lemma}{Lemma}

\newtheorem{corollary}{Corollary}
\theoremstyle{definition}

\newcommand{\norm}[1]{\left\lVert#1\right\rVert}
\newcommand{\vb}[1]{\vec{\mathbf{#1}}}
\allowdisplaybreaks

\makeatletter
\@namedef{subjclassname@2020}{%
  \textup{2020} Mathematics Subject Classification}
\makeatother

\DeclareRobustCommand{\topbot}{\genfrac{}{}{0pt}{}}
\newcommand{\C}{\mathbb{C}}
\newcommand{\R}{\mathbb{R}}
\newcommand{\Q}{\mathbb{Q}}
\newcommand{\N}{\mathbb{N}}
\newcommand{\Z}{\mathbb{Z}}

\newcommand{\SL}{\textnormal{SL}}
\title[Mean Values and Diophantine Approximation in Number Fields]{Mean Values over Lattices in Number Fields and Effective Diophantine Approximation}

\author{Nathan Hughes}

\address{Department of Mathematics, University of Exeter, Exeter, EX4 4QF, UK}
\email{nh477@exeter.ac.uk}
\begin{document}
	\maketitle
	\begin{abstract}
		\centering We answer a question raised by Alam and Ghosh concerning an error term for a spiralling result in Diophantine approximation by rationals in a number field. The proof relies on a generalisation of Rogers' Mean Value Theorem to algebraic number fields and an effective ergodic theorem due to Gaposhkin.
	\end{abstract}
	\section{Introduction}
	In the past decade, there has been interest in quantitative results of Dirichlet's Approximation Theorem using dynamical methods. In \cite{athreyaErgodicTheoryDiophantine2016}, Athreya, Parrish and Tseng showed that the number of approximates (tuples $(\mathbf{p},q)\in  \Z^d\times \N$ satisfying $\norm{q\mathbf{x}-\mathbf{p}}<cq^{-{1}/{d}}$) to almost all vectors $\mathbf{x}\in \R^d$ with $1\leq q \leq T$ grows asymptotically with $\log T$. In \cite{athreyaSpiralingApproximationsSpherical2014}, the equidistribution of the approximates' directions was established. 
	By projecting a solution $(\mathbf{p},q)$ to the unit sphere $\mathbb{S}^{d-1}$ via the map $\theta_\mathbf{x}(\mathbf{p},q) = \tfrac{1}{\norm{q\mathbf{x}-\mathbf{p}}}(q\mathbf{x}-\mathbf{p})$, it was shown that the proportion of approximates pointing in a direction $A \subseteq \mathbb{S}^{d-1}$ is $\text{vol}(A)$, where $A$ is a set with boundary of measure zero and $\text{vol}$ is the Lebesgue probability measure on $\mathbb{S}^{d-1}$. A weighted version of this result was derived by Kleinbock, Shi and Weiss in \cite{kleinbockPointwiseEquidistributionError2017}, along with an error rate for the count of all approximates. 
	
	Some of these results were generalised to number fields by Alam and Ghosh in \cite{alamEquidistributionHomogeneousSpaces2020}, where they investigated weighted approximation by elements of rings of integers and the distribution of the approximates' directions. These results are obtained by applying Birkhoff's Ergodic Theorem to diagonal flows on a suitable space of lattices, counting the number of lattice points in a set, and using a Siegel-type formula relating integrals over the space of lattices to an integral over the underlying space. Alam and Ghosh also expressed interest in obtaining an error term for these results. 
	
	In this paper, we provide an error term by using a number field analogue of Rogers' moment formula (Theorem \ref{Rogers3}) in tandem with an effective ergodic theorem due to Gaposhkin (Corollary \ref{gaposhkinCorr}). Moment formulas for integrals over lattice spaces have been used by many authors to obtain error terms for quantitative number theory problems. Rogers' original moment formula \cite{rogersMeanValuesSpace1955} was used by Schmidt to prove lattice counting results in very general subsets of $\R^n$ in \cite{schmidtMETRICALTHEOREMGEOMETRY}. More recently, in \cite{hanRogersMeanValue2022}, for $S$ a finite subset of valuations on $\Q$, Han proved a higher moment formula and deduced an effective $S$-arithmetic version of both the Oppenheim Conjecture and Gauss' Circle Problem for star-shaped sets. Similar estimations in the adelic case were recently announced in a preprint by Kim \cite{kimAdelicRogersIntegral2022}.
	
	\subsection{Main Results}
	Let $K$ be an algebraic number field with ring of integers $\mathcal{O}_K$ and set of Archimedean non-conjugate embeddings $S$. To each $\nu\in S$ we write $K_\nu$ to denote the completion of $K$ with respect to $\nu$ and $\iota_\nu:K \hookrightarrow K_\nu$ the associated embedding. For $x \in K$, let $\iota_S(x) = \left(\iota_\nu(x)\right)_{\nu\in S}$ be the twisted diagonal embedding, whose image is a subset of $K_S := K \otimes_\Q \R$.
	
	Given $d\geq 2$, let $G =\prod_{\nu\in S}\SL_d(K_\nu)$ and $\Gamma=\SL_d(\iota_S(\mathcal{O}_K))$. Then $\Gamma$ is a lattice in $G$ and there exists a left-$G$-invariant probability measure $\mu$ on $X:=G/\Gamma$. Let $\lambda$ be the Haar measure on $K_S$, normalised so that $\mathcal{O}_K$ is a lattice of unit covolume. We will write $\Lambda$ to represent an arbitrary element in $X$, which corresponds to a lattice in $K_S^d$ that is a product of lattices in $K_\nu^d$ for each $\nu\in S$, and denote by $\Lambda_0 \in X$ the standard lattice in $K_S^d$.
	\subsubsection{Rogers' Theorem over Number Fields}
	Our first result is a Rogers type formula over $K_S^d$. The proof closely follows Schmidt's argument of the formula in \cite{schmidtMittelwerteBerGitter1957}, which is the case $K=\Q$.
	\begin{theorem}\label{Rogers3}
		Let $1\leq k \leq d-1$ and $f\in L^1((K_S^d)^k)$ be non-negative. Then
		\begin{gather*}
			\int_{X} \sum_{\vb{x}_1,\dots,\vb{x}_k \in \Lambda }f(\vb{x}_1,\dots,\vb{x}_k)\,d\mu = f(0,\dots,0) + \int_{(K_S^d)^k}f(\vb{x}_1,\dots,\vb{x}_k)\,d\lambda^k \\
			+ \sum_{r=1}^{k-1} \sum_{D \in \mathcal{D}_{r,k}} \frac{1}{\textnormal{covol}(\Phi(D))} \int_{(K_S^d)^r} f\left(\left(\vb{x}_1,\dots,\vb{x}_r\right)\cdot\iota_S(D)\right) d\lambda^r
		\end{gather*}
		where $\lambda^r$ is the product measure on $(K_S^d)^r$, $\mathcal{D}_{r,k}$ is the set of all rank $r$ row-reduced echelon matrices in $\textnormal{Mat}_{r\times k}(K)$ and for $D \in \mathcal{D}_{r,k}$,
		\begin{equation*}
			\Phi(D) = \left\{ \vb{V} = (\vb{x}_1,\dots,\vb{x}_r)\in \Lambda_0^r\,:\, \textnormal{rank}(\vb{V})=r,\, \vb{V}\cdot \iota_S(D)\in \Lambda_0^k \right\}
		\end{equation*}
	\end{theorem}
	\subsubsection{Effective Diophantine Approximation in Number Fields}
	For $\nu\in S$, let $d_\nu =1$ if $\nu$ is a real valuation and $d_\nu=2$ otherwise. Let $m,n$ be positive integers satisfying $d=m+n$ and let $\mathbf{a} \in \R_{>0}^{m\cdot \#S},\mathbf{b}\in \R_{>0}^{n \cdot\#S}$ be weight vectors such that
	\begin{equation*}
		\sum_{i=1}^m \sum_{\nu\in S}d_\nu a_{i\nu} = \sum_{j=1}^n \sum_{\nu\in S}d_\nu b_{j \nu} = 1
	\end{equation*}
	To these weights we associate functions $\norm{\cdot}_\mathbf{a} : K_S^m \rightarrow \R_{\geq 0}$, $\norm{\cdot}_{\mathbf{b}} : K_S^n \rightarrow \R_{\geq 0}$, defined by
	\begin{equation*}
		\norm{\vb{x}}_\mathbf{a} = \max_{\stackrel{\nu\in S}{1\leq i \leq m}}|x_{i\nu}|^\frac{1}{a_{i\nu}},\,\quad \norm{\vb{y}}_\mathbf{b} = \max_{\stackrel{\nu\in S}{1\leq j \leq n}}|y_{i\nu}|^\frac{1}{b_{j\nu}}
	\end{equation*}
	which will serve as analogues for norms on $K_S^m$ and $K_S^n$ respectively. We wish to count the number of approximates to $\vartheta \in \text{Mat}_{m\times n}(K_S)$ of the form $(\vb{p},\vb{q})\in \iota_S(\mathcal{O}_K)^m \times \iota_S(\mathcal{O}_K)^n$ satisfying the inequalities
	\begin{equation}\label{Dioph}
		\norm{\vartheta\vb{q} - \vb{p}}_\mathbf{a}\norm{\vb{q}}_\mathbf{b} < c , \quad 1\leq \norm{\vb{q}}<e^T
	\end{equation} 
	For $T,c>0$, define the set
	\begin{equation*}
		E_{T,c} = \left\{(\vb{x},\vb{y})\in K_S^m\times K_S^n\,:\, \norm{\vb{x}}_\mathbf{a}\norm{\vb{y}}_\mathbf{b}<c , \, 1 \leq \norm{\vb{y}}< e^T\right\}
	\end{equation*}
	We associate to $\vartheta$ the lattice 
	\begin{equation*}
	\Lambda_\vartheta =\begin{pmatrix} \text{Id} & \vartheta \\ 0 &\text{Id}\end{pmatrix}\Lambda_0
	\end{equation*} 
	It is easy to see that $\#(E_{T,c}\cap \Lambda_\vartheta)$ counts the number of solutions to the inequalities in (\ref{Dioph}). Our next theorem provides an effective count of these solutions in terms of $T$.
	\begin{theorem}\label{maintheorem}
		Let $d\geq 3$ and $c>0$. Then for $\mu$-almost every $\Lambda\in X$, any $\epsilon >0$, and sufficiently large $T>0$,
		\begin{equation*}
			\left|\#(E_{T,c}\cap \Lambda) - \lambda(E_{T,c}) \right| = o\left(T^{\frac{1}{2}}(\log T)^\frac{3}{2}(\log \log T)^{\frac{1}{2}+\epsilon}\right)
		\end{equation*}
		In particular, for almost every $\nu \in \textnormal{Mat}_{m\times n }(K_S)$ and for every $\epsilon >0$,
		\begin{equation*}
			\left|\#(E_{T,c}\cap \Lambda_\vartheta) - \lambda(E_{T,c}) \right| = o\left(T^{\frac{1}{2}}(\log T)^\frac{3}{2}(\log \log T)^{\frac{1}{2}+\epsilon}\right)
		\end{equation*}
	\end{theorem}
	The theorem above can be generalised to an effective spiralling result for number fields. Let
	\begin{equation*}
		\mathbb{S}^{m\deg K -1} = \left\{\vb{x}\in K_S^m\,:\, \sum_{i=1}^m\sum_{\nu\in S}|x_{i\nu}|^2=1\right\}
	\end{equation*}
	and $\pi_\mathbf{a}(\vb{x}) = \left\{(e^{ta_{1\nu}}x_{1\nu},\dots,e^{ta_{m\nu}}x_{m\nu} )_{\nu\in S}\,:\, t \in \R\right\}\cap \mathbb{S}^{m\deg K -1}$ be the $\mathbf{a}$-weighted projection of $\vb{x}$ onto $\mathbb{S}^{m\deg K-1}$. Similarly define the $\mathbf{b}$-weighted projection $ \pi_\mathbf{b}:K_S^n\rightarrow \mathbb{S}^{n \deg K-1}$. Let
	\begin{equation*}
		E_{T,c}(A,B) = \left\{(\vb{x},\vb{y})\in K_S^m \times K_S^n \,:\, \topbot{ \norm{\vb{x}}_\mathbf{a}\norm{\vb{y}}_\mathbf{b} < c, \, 1 \leq \norm{\vb{y}}_\mathbf{b} < e^T, }{ \pi_\mathbf{a}(\vb{x}) \in A, \, \pi_\mathbf{b}(\vb{y})\in B}\right\}
	\end{equation*}
	Our final theorem gives an effective count of the distribution of lattice points with respect to these projections and the corresponding result for measuring the distribution of Diophantine approximates follows.
	\begin{theorem}\label{spiralling}
		Let $d\geq 3$ and $c>0$. Then for $\mu$-almost every $\Lambda\in X$, measurable $A\subseteq \mathbb{S}^{m\deg K-1}$,$B \subseteq\mathbb{S}^{n\deg K -1}$ with measure $0$ boundary, any $\epsilon >0$, and sufficiently large $T>0$,
		\begin{equation*}
			\left|\#(E_{T,c}(A,B)\cap \Lambda) - \lambda(E_{T,c}(A,B)) \right| = o\left(T^{\frac{1}{2}}(\log T)^\frac{3}{2}(\log \log T)^{\frac{1}{2}+\epsilon}\right)
		\end{equation*}
		In particular, for almost every $\vartheta \in \textnormal{Mat}_{m\times n }(K_S)$ and for every $\epsilon >0$,
		\begin{equation*}
			\left|\#(E_{T,c}(A,B)\cap \Lambda_\vartheta) - \lambda(E_{T,c}(A,B)) \right| = o\left(T^{\frac{1}{2}}(\log T)^\frac{3}{2}(\log \log T)^{\frac{1}{2}+\epsilon}\right)
		\end{equation*}
	\end{theorem}
	Note that each of these results only hold in sufficiently high dimensions. In particular, the case $k=2$ of Theorem \ref{Rogers3}, which we use to prove Theorems \ref{maintheorem} and \ref{spiralling}, only hold for $d\geq 3$. To obtain formulas for these higher moments (or lower dimensions), different techniques are needed; some results in this direction are described by Kelmer and Yu \cite{kelmerSecondMomentSiegel2019} and Kleinbock and Yu \cite{kleinbockDynamicalBorelCantelliLemma2020}.
	\section{Higher Moment Formulas for Number Fields}
	In this section we will prove theorem \ref{Rogers3}, which describes the integral 
	\begin{equation*}
		\int_X \sum_{\vb{x}_1,\dots,\vb{x}_k\in \Lambda} h(\vb{x}_1,\dots,\vb{x}_k)\,d\mu(\Lambda)
	\end{equation*}
	in terms of integrals over $(K_S^d)^r$ where $0\leq r \leq k\leq d-1$ and $h\in L^1((K_S^d)^k)$ is non-negative. By setting $h= \prod \widehat{f}$, we recover an expression for higher moments of $\widehat{f}$, which we will use to prove the effective counting and spiralling results.
	\subsection{Linear Independence of Lattice Points}
	For a lattice $\Lambda \in X$, let $1 \leq k < d$ and 
	\begin{equation*}
		\Lambda^k = \left\{ (\vb{x}_1,\dots,\vb{x}_k)\, : \, \vb{x}_i \in \Lambda \right\}
	\end{equation*}
	Suppose $\vb{V}=\left(\vb{x}_1,\dots,\vb{x}_k\right) \in \Lambda^k$  are vectors with $K$-rank $r$. Let $1 \leq j_1, \dots, j_r \leq k$ be the indeces of a maximal subset of vectors $\vb{x}_{j_i}$ that is $K$-linearly independent. For all other indeces $j \neq j_i$, let $s_j = \max_{i}\{j_i \,:\, j_i \leq j \}$ and write $\vb{x}_j = \sum_{i=1}^{s_j} k_{j,i} \vb{x}_{j_i}$. Following the notation of \cite{hanRogersMeanValue2022}, to $\vb{V}$ we associate a matrix $D \in \text{Mat}_{r\times k}(K)$ such that
	\begin{enumerate}
		\item\label{conds} $([D]^{j_1},\dots,[D]^{j_r}) = \text{Id}_r$,
		\item\label{conds2} for any $j \neq j_i$, $1\leq i \leq r$, $[D]^j = (k_{j,1},\dots,k_{j,s_j},0\dots,0)^\text{T}$, $k_{j,l}\in K$,
	\end{enumerate}
	where $[D]^j$ denotes the $j^\text{th}$ column of $D$. The matrix $D$ acts on $r$-tuples of linearly independent sets of vectors to give $k$-tuples of linearly dependent vectors. Define the set 
	\begin{equation*}
		\Phi_\Lambda(D) = \left\{ \vb{V}=\left(\vb{x}_1,\dots,\vb{x}_r\right)\in\Lambda_0^r \, : \, \text{rank}(\vb{V})=r,\, \vb{V}\cdot \iota_S(D) \in \Lambda_0^k \right\}
	\end{equation*}
	Let $\mathcal{D}_{r, k}$ be the set of all rank $r$ row-reduced echelon matrices in $\text{Mat}_{r\times k}(K)$. Let us also define, for $c \in K^\times$,
	\begin{equation*}
		N(c) = \prod_{\nu\in S}|\iota_\nu(c)|^{d_\nu}
	\end{equation*}
	Multiplication by $c$ on $K_S^d$ then corresponds to a linear transformation of determinant $N(c)^d$.
	\begin{lemma}\label{lattice}
		For any $D \in \mathcal{D}_{r,k}$, $\Phi_\Lambda(D)$ is a lattice in $\left(K_S^d\right)^r$ and $\textnormal{covol}(\Phi_{\Lambda}(D)) = \textnormal{covol}(\Phi_{\Lambda_0}(D))$ for all $\Lambda \in X$. In particular, if $D=\left(1,\tfrac{p}{q}\right)\in \mathcal{D}_{1,2}$ with $p,q$ coprime, then $\textnormal{covol}(\Phi_{\Lambda}(D)) = N(q)^d$.
	\end{lemma}
	\begin{proof}
		Clearly $\Phi_\Lambda(D)$ is an additive subgroup of $\Lambda^r$. For each entry $k_{i,j}\neq0,1$ of $D$, let $\left(k_{i,j}\right)\subset K$ be the fractional ideal generated by $k_{i,j}$ and define $\mathfrak{k}=\sum_{k_{i,j}\neq0,1} \left(k_{i,j}\right)$. There exists $\alpha \in \mathcal{O}_K$ such that $\alpha \mathfrak{k} \lhd \mathcal{O}_K$. We then have $\left(\alpha\Lambda^d\right)^r \leq \Phi_\Lambda(D) \leq \left(\Lambda^d\right)^r$, which shows $\Phi_\Lambda(D)$ is a lattice. Since $\Lambda = g\Lambda_0$ for some $g\in G$ with $\det g_\nu = 1$ for all $\nu\in S$, we must have $\textnormal{covol}(\Phi_{\Lambda}(D)) = \textnormal{covol}(\Phi_{\Lambda}(D))$.
	
	Given ${D}=(1,\tfrac{p}{q})\in \mathcal{D}_{1,2}$ where $p,q\in \mathcal{O}_K$ are coprime, the ideal $\mathfrak{k} = \left(\frac{p}{q}\right)$ is principal and
	\begin{equation*}
		\Phi_\Lambda({D}) = \left\{\vb{x} \in \Lambda \, : \, \tfrac{p}{q}\cdot\vb{x} \in \Lambda\right\}
	\end{equation*}
	Clearly $q\cdot\Lambda \subset \Phi_\Lambda({D})$. Conversely, $\vb{x} \in \Phi_\Lambda({D})$ implies $p\cdot\vb{x} \in q\cdot\Lambda$, and since $p,q$ are coprime, $\vb{x} \in q\cdot\Lambda$. Therefore $q\cdot\Lambda = \Phi_\Lambda({D})$ and $\text{covol}_{\Phi_\Lambda}(D) = N(q)^d$.
	\end{proof}	
	We can write
	\begin{equation}\label{decomp}
		\Lambda^k = \bigsqcup_{r=0}^k\left\{\vb{V} \in \Lambda^k \,:\, \text{rank}(\vb{V}) = r\right\}
	\end{equation}
	Clearly for $r=0$ the set comprises of only $\{(0,\dots,0)\}$ and the case $r=k$ is the set of all elements in $\Lambda^k$ that are linearly independent in $\Lambda$. We may then rewrite equation (\ref{decomp}) as
	\begin{gather}
		\Lambda^k = \{(\vb{0},\dots, \vb{0})\} \sqcup \{(\vb{x}_1,\dots,\vb{x}_k)\in \Lambda^k\,\, \text{lin. ind.}\} \nonumber\\
		\sqcup \bigsqcup_{r=1}^{k-1} \bigsqcup_{D \in \mathcal{D}_{r,k}} \left\{(\vb{x}_1,\dots,\vb{x}_r)\cdot\iota_S(D) \, : \, (\vb{x}_1,\dots,\vb{x}_r) \in \Phi_\Lambda(D)\right\}\label{decomp2}
	\end{gather}
	\subsection{Higher Moment Mean Value Formulas on $K_S^d$}
	A theorem of Weil implies the following mean value formula.
	\begin{theorem}[{{{Weil \cite{weilQuelquesResultatsSiegel1946}}}}]
		\label{Weil}
		Let $f\in L^1(K_S^d)$ be non-negative. Then
		\begin{equation*}
			\int_{X} \widehat{f}(\Lambda) \,d\mu(\Lambda) = \int_{K_S^d} f(\vb{x}) \,d\lambda(\vb{x})
		\end{equation*}
	\end{theorem}
	We wish to generalise this formula to sums over linearly independent sets of vectors in $K_S^d$. To do so, first we need to introduce some subspaces of $G$. Let $1 < r < d-1$ and define
	\begin{equation*}
		\mathcal{C}_r = \left\{ \begin{pmatrix}
			\text{Id} & A \\ 0 & B 
		\end{pmatrix} \,:\, A \in \text{Mat}_{r\times d-r}(K_S),\, B \in \SL_{d-r}(K_S) \right\}\subset G
	\end{equation*}
	We can decompose $\mathcal{C}_r = \mathcal{B}_r\times \mathcal{A}_r$, where
	\begin{equation*}
		\mathcal{A}_r = \left\{\begin{pmatrix}
			\text{Id} & A\\0 & \text{Id}
		\end{pmatrix} \,:\, A \in \text{Mat}_{r\times d-r}(K_S) \right\} 
	\end{equation*}  
	\begin{equation*}
		\mathcal{B}_r = \left\{ \begin{pmatrix}\text{Id} & 0 \\ 0 & B\end{pmatrix} \,: \, B \in \SL_{d-r}(K_S)\right\} 
	\end{equation*}
	In fact, $\mathcal{C}_r \cong \mathcal{B}_r \times \mathcal{A}_r$. Let $\nu_{\mathcal{A},r}$, $\nu_{\mathcal{B}_r}$ be Haar measures on $\mathcal{B}_r$ and $\mathcal{A}_r$ respectively. Since $\mathcal{B}_r \cong \SL_{d-r}(K_S)$, the measure $\nu_{\mathcal{B}_r}$ induces a probability measure on the quotient $\mathcal{B}_r/(\mathcal{B}_r \cap \Gamma)$, hence a finite volume fundamental domain $\mathcal{F}_{\mathcal{B}_r}\subset \mathcal{B}_r$ exists. Similarly, $\mathcal{A}_r / (\mathcal{A}_r\cap \Gamma) \cong K_S^{r(d-r)}/\iota_S(\mathcal{O}_K)^{r(d-r)}$, so a finite volume fundamental domain $\mathcal{F}_{\mathcal{A}_r}\subset \mathcal{A}_r$ exists for this space also.
	
	Let $\nu_\mathcal{C}$ be the product Haar measure on $\mathcal{C}_r$. Since $\mathcal{C}_r$ is isomorphic to the direct product $\mathcal{B}_r\times\mathcal{A}_r$, it is also unimodular so contains a lattice. The homeomorphism
	\begin{equation*}
		\mathcal{B}_r/(\mathcal{B}_r\cap \Gamma)\times \mathcal{A}_r/(\mathcal{A}_r\cap \Gamma) \leftrightarrow \left(\mathcal{B}_r\times\mathcal{A}_r \right)/\left(\left(\mathcal{B}_r\times\mathcal{A}_r\right)\cap \Gamma\right) \cong \mathcal{C}_r/(\mathcal{C}_r\cap \Gamma)
	\end{equation*}
	\begin{equation*}
		\left(b\left(\mathcal{B}_r \cap \Gamma\right),a\left(\mathcal{A}_r \cap \Gamma\right)\right)\mapsto ba\left(\mathcal{C}_r \cap \Gamma\right)
	\end{equation*}
	shows that $\mathcal{F}_{\mathcal{B}_r}\times\mathcal{F}_{\mathcal{A}_r}$ is homeomorphic to a fundamental domain of $\mathcal{C}_r/(\mathcal{C}_r\cap \Gamma)$, denoted $\mathcal{F}_{\mathcal{C}_r}$. The above homeomorphism is induced by the measure-preserving homeomorphism $\mathcal{B}_r \times \mathcal{A}_r \rightarrow \mathcal{C}_r$, $(b,a)\mapsto ba$, hence $\nu_{\mathcal{C}_r}(\mathcal{F}_{\mathcal{C}_r}) =\nu_{\mathcal{A}_r}(\mathcal{F}_{\mathcal{A}_r})\nu_{\mathcal{B}_r}(\mathcal{F}_{\mathcal{B}_r}) = 1$, where $\nu_{\mathcal{C}_r}$ is the product measure on $\mathcal{B}_r\times\mathcal{A}_r$.
	
	The following theorem is a generalisation of Theorem \ref{Weil}. The strategy closely follows the proof given by Schmidt for the case $K=\Q$ \cite{schmidtMittelwerteBerGitter1957}. The proof relies on induction, where we transform the linearly independent vectors so that the first $k-1$ vectors lie in a particular subspace of $K_S^d$. The properties of the transformations in $\mathcal{C}_r$ acting on these vectors will then recover the mean value properties we require.
	\begin{theorem}\label{linind}
		Let $f\in L^1(\left(K_S^d\right)^k)$ be non-negative. Then
		\begin{equation*}
			\int_X \sum_{\stackrel{\vb{v}_1,\dots,\vb{v}_k\in\Lambda}{\textnormal{lin.ind.}}}f(\vb{v}_1,\dots,\vb{v}_k)\,d\mu(\Lambda) = \underbrace{\int_{K_S^d}\cdots \int_{K_S^d}}_{k\text{ times}} f(\vb{x}_1,\dots,\vb{x}_k)\,d\lambda(\vb{x}_1)\dots d\lambda(\vb{x}_k)
		\end{equation*}
	\end{theorem}
	\begin{proof}
		The case $k=1$ is Theorem \ref{Weil}. We proceed by induction. Suppose the theorem is true for $r<d-2$. Let
		\begin{equation*}
			\mathcal{V}_{r+1}(\Lambda) = \left\{\vb{V} \in \Lambda^{r+1} \,:\, \text{rank}\left(\vb{V}\right)=r+1\right\}
		\end{equation*}
		and for $\vb{V}_{r} \in \mathcal{V}_{r}(\Lambda)$, let
		\begin{equation*}
			L_\Lambda(\vb{V}_{r}) = \left\{\vb{v}\in \Lambda\,:\,\text{rank}\left(\vb{V}_{r},\vb{v}\right)=r+1\right\}
		\end{equation*}
		When no confusion arises, we will denote this set by $L_\Lambda$. Then we calculate
		\begin{gather*}
			\int_X \sum_{\vb{v} \in \mathcal{V}_{r+1}(\Lambda)}f(\vb{v})\,d\mu(\Lambda) \\
			= \int_X \sum_{\vb{V}_{r}\in \mathcal{V}_{r}(\Lambda)} \sum_{{\vb{v}\in L_\Lambda(\vb{V}_{r})}}f(\vb{V}_{r},\vb{v})\,d\mu(\Lambda) \\
			= \int_{\mathcal{F}_X} \sum_{\vb{Z}_{r}\in\mathcal{V}_{r}(\Lambda_0)} \sum_{\vb{z} \in L_{\Lambda_0}(\vb{Z}_{r})} f(g\vb{Z}_{r},g\vb{z})\,d\mu(g) \\
			= \sum_{\vb{Z}_{r}\in\mathcal{V}_{r}(\Lambda_0)}\int_{\mathcal{F}_X}  \sum_{\vb{z} \in L_{\Lambda_0}(\vb{Z}_{r})} f(g\vb{Z}_{r},g\vb{z})\,d\mu(g)
		\end{gather*}
		where $\mathcal{F}_X$ is a fundamental domain for $X$ in $G$. Given linearly independent $\vb{Z}_r \in \mathcal{V}_r(\Lambda_0)$, there exists a matrix $u \in \mathcal{C}_{r}$ such that $\text{span}_K(u\vb{Z}_r) = K_S^{r} \times \{\vb{0}\}^{d-r}=:E$. Given such a transformation $u$, if $\vb{Z}_r,\vb{z}$ are linearly independent, then $u\vb{z} \notin E$. Since $\mu$ is $G$-invariant, for fixed $\vb{Z}_r$ and $u$ as above,
		\begin{gather*}
			\int_{\mathcal{F}_X} \sum_{\vb{z} \in L_{\Lambda_0}} f(g\vb{Z}_r,g\vb{z})\,d\mu(g) \\
			=\int_{\mathcal{F}_X} \sum_{\vb{z} \in L_{\Lambda_0}} f(gu\vb{Z}_r,gu\vb{z})\,d\mu(g) \\
			=\int_{\mathcal{F}_X} \sum_{\vb{z} \notin E} f(gu\vb{Z}_r,g\vb{z})\,d\mu(g) 
		\end{gather*}
		Let $c \in \mathcal{C}_{r}$. Then $cu\vb{Z}_r = u\vb{Z}_r$. Again, since $\mu$ is $G$-invariant,
		\begin{gather*}
			\int_{\mathcal{F}_X} \sum_{\vb{z} \notin E} f(gu\vb{Z}_r,g\vb{z})\,d\mu(g) \\
			= \int_{\mathcal{F}_X} \sum_{\vb{z} \notin E} f(gcu\vb{Z}_r,gc\vb{z})\,d\mu(g)  \\
			=\int_{\mathcal{F}_X} \sum_{\vb{z} \notin E} f(gu\vb{Z}_r,gc\vb{z})\,d\mu(g) 
		\end{gather*}
		Write $c=ba$, where $a\in\mathcal{A}_r$, $b\in\mathcal{B}_r$. Then
		\begin{gather*}
			\int_{\mathcal{F}_X} \sum_{\vb{z} \in L_{\Lambda_0}} f(g\vb{Z}_r,g\vb{z})\,d\mu(g) \\
			= \int_{\mathcal{F}_{\mathcal{C}_{r}}} \int_{\mathcal{F}_X} \sum_{\vb{z} \in L_{\Lambda_0}} f(g\vb{Z}_r,g\vb{z})\,d\mu(g) \, d\nu_{\mathcal{C}_{r}}(c) \\
			= \int_{\mathcal{F}_{\mathcal{C}_{r}}}\int_{\mathcal{F}_X} \sum_{\vb{z} \notin E} f(gu\vb{Z}_r,gc\vb{z})\,d\mu(g) \,d\nu_{\mathcal{C}_{r}}(c)\\
			= \int_{\mathcal{F}_{\mathcal{A}_{r}}} \int_{\mathcal{F}_{\mathcal{B}_{r}}}  \int_{\mathcal{F}_X} \sum_{\vb{z} \notin E} f(gu\vb{Z}_r,gba\vb{z})\,d\mu(g)   \,d\nu_{\mathcal{B}_{r}}(b)\,d\nu_{\mathcal{A}_{r}}(a) \\
			\int_{\mathcal{F}_X} \int_{\mathcal{F}_{\mathcal{B}_{r}}} \int_{\mathcal{F}_{\mathcal{A}_{r}}}    \sum_{\vb{z} \notin E} f(gu\vb{Z}_r,gba\vb{z}) \,d\nu_{\mathcal{A}_{r}}(a) \,d\nu_{\mathcal{B}_{r}}(b) \,d\mu(g) 
		\end{gather*} 
		where the last equality is by Fubini's Theorem. 
		Notice that
		\begin{equation*}
			a\vb{z} = \begin{pmatrix}
				\text{Id} & A \\ 0 & \text{Id}
			\end{pmatrix} \begin{pmatrix}
				(\mathbf{z}_i)_{i=1}^{r} \\ (\mathbf{z}_i)_{i=r+1}^d
			\end{pmatrix} = \begin{pmatrix}
				(\mathbf{z}_i)_{i=1}^{r} + A(\mathbf{z}_i)_{i=1}^{r} \\
				(\mathbf{z}_i)_{i=r+1}^d
			\end{pmatrix}
		\end{equation*}
		Therefore for any integrable function $h:K_S^d\rightarrow \R_{\geq 0}$,
		\begin{gather*}
			\int_{\mathcal{F}_\mathcal{A}} \sum_{\vb{z}\notin E}h(a \vb{z}) \,d\nu_{\mathcal{A}_{r}} (a)\\ = \int_{K_S}\cdots\int_{K_S}\sum_{(\mathbf{z}_{r+1},\dots,\mathbf{z}_d)\neq \vb{0}} h(\mathbf{x}_1,\dots,\mathbf{x}_{r},\mathbf{z}_{r+1} ,\dots,\mathbf{z}_d)\, d\lambda_{K_S}(\mathbf{x}_1)\dots d\lambda_{K_S}(\mathbf{x}_{r})
		\end{gather*}
		and by Theorem \ref{Weil},
		\begin{equation*}
			\int_{\mathcal{F}_\mathcal{B}}\int_{\mathcal{F}_\mathcal{A}} \sum_{\vb{z}\notin E}h(ba\vb{z})d\nu_{\mathcal{A}_{r}}(a) d\nu_{\mathcal{B}_{r}}(b) = \int_{K_S^d}h(\vb{x})\,d\lambda(\vb{x})
		\end{equation*}
		Therefore
		\begin{gather*}
			\int_{\mathcal{F}_X} \int_{\mathcal{F}_{\mathcal{B}_{r}}} \int_{\mathcal{F}_{\mathcal{A}_{r}}}    \sum_{\vb{z} \notin E} f(gu\vb{Z}_r,gba\vb{z}) \,d\nu_{\mathcal{A}_{r}}(a) \,d\nu_{\mathcal{B}_{r}}(b) \,d\mu(g) \\
			=\int_{\mathcal{F}_X} \int_{K_S^d}  f(gu\vb{Z}_r,g\vb{x})\,d\lambda(\vb{x})  \,d\mu(g) 
		\end{gather*}
		The theorem  follows by induction.
	\end{proof}
	
	We can now prove Theorem \ref{Rogers3}, a number field analogue of Rogers' Theorem. Again, the result can be obtained in a similar manner to Schmidt's proof of Rogers' Theorem for the case $K=\Q$ \cite{schmidtMittelwerteBerGitter1957}.
	
	\begin{proof}[proof of Theorem \ref{Rogers3}]
		
		By equation (\ref{decomp2}), we can write
		\begin{gather}
			\int_{X} \sum_{\vb{x}_1,\dots,\vb{x}_k \in \Lambda }f(\vb{x}_1,\dots,\vb{x}_k)d\mu = f(0,\dots,0) + \int_{X} \sum_{\substack{\vb{x}_1,\dots,\vb{x}_k \in \Lambda \\ \text{lin. ind.}}} f(\vb{x}_1,\dots,\vb{x}_k) d\mu\nonumber \\
			+ \sum_{r=1}^{k-1}\sum_{D \in \mathcal{D}_{r,k}} \int_{X} \sum_{(\vb{x}_1,\dots,\vb{x}_r) \in \Phi_\Lambda(D)} f\left( (\vb{x}_1,\dots,\vb{x}_r)\cdot\iota_S(D) \right) d\mu \label{partial}
		\end{gather}
		By Lemma \ref{lattice}, each set $\Phi_\Lambda(D)$ defines a lattice in $(K_S^r)^d$. For any $D \in \mathcal{D}_{r,k}$, write $\bar{f}\left(\vb{z}_1,\dots,\vb{z}_r\right):=f\left((\vb{z}_1,\dots,\vb{z}_r)\cdot\iota_S(D)\right)$. Then by Theorem \ref{linind}, 
		\begin{gather*}
			\int_{\mathcal{F}_X} \sum_{\vb{z}_1,\dots,\vb{z}_r \in \Phi_{g\Lambda_0}(D)} \bar{f}\left(\vb{z}_1,\dots,\vb{z}_r\right)\,d\mu(g)\\
			= \frac{1}{\text{covol}(\Phi_{\Lambda}(D))}\int_{K_S^d}\cdots \int_{K_S^d} \bar{f}\left(\vb{x}_1,\dots,\vb{x}_r\right) \,d\lambda(\vb{x}_1)\dots d\lambda(\vb{x}_r) \\
			= \frac{1}{\text{covol}(\Phi(D))} \int_{K_S^d}\cdots \int_{K_S^d} f\left((\vb{x}_1,\dots,\vb{x}_r)\cdot \iota_S(D)\right) \,d\lambda(\vb{x}_1)\dots d\lambda(\vb{x}_r) 
		\end{gather*}
		This combined with equation (\ref{partial}) gives the result.
	\end{proof}
	
	\section{Effective Diophantine Approximation In Number Fields}
	In this section we will prove Theorem \ref{maintheorem} and Theorem \ref{spiralling}. First we recall Gaposhkin's Effective Ergodic Theorem and one of its corollaries. The proof of these theorems then follows from a calculation using Theorem \ref{Rogers3}. 
	\begin{theorem}[{{\cite[Theorem 1]{gaposhkinEstimatesMeansAlmost1979}}}]
		Let $\xi:X\times\R\rightarrow \C$ be a mean-square continuous stationary processes on a measure space $(X,\mu)$ such that
		\begin{equation*}
			\int_X \xi(x,t)\,d\mu(x) = \mu(\xi), \quad \int_X |\xi(x,t)|^2\,d\mu(x) >0, 
		\end{equation*}
		Define \begin{equation*}
			R^*(\tau) = \int_X \left(\frac{1}{\tau}\int_0^\tau \xi(x,t)\,dt\right)^2\,d\mu(x) - \mu(\xi)^2
		\end{equation*}
		Then for any monotonically increasing function $\Psi:(1,\infty)\rightarrow \R_{\geq 0}$ such that $\int_1^\infty \left(t\Psi(t)\right)^{-1}\,dt<\infty$, and for $\mu$-almost every $x \in X$,
		\begin{equation*}
			\frac{1}{T}\int_0^T \xi(x,t) \,dt = \mu(\xi) + o\left(\sqrt{\frac{\Psi(T)}{T}}\int_1^T \sqrt{\frac{R^*(\tau)}{\tau}}\,d\tau\right)
		\end{equation*}
	\end{theorem}
	Using this theorem, Gaposhkin calculated various error rates for processes with $R^*$ asymptotically shrinking with $T^{-\alpha}\left(\log T\right)^{-\beta}\left(\log \log T\right)^{-\gamma}$, one of which is written below. 
	\begin{corollary}[{{\cite[Theorem 4(vii)]{gaposhkinDependenceConvergenceRate1982}}}]\label{gaposhkinCorr}
		Suppose $R^*(\tau) = \frac{1}{\tau}$. Then for any $\epsilon >0$,
		\begin{equation*}
			\frac{1}{T}\int_0^T \xi(x,t)\,dt = \mu(\xi) + o\left(T^{-\frac{1}{2}}\left(\log T\right)^{\frac{3}{2}}\left(\log \log T\right)^{\frac{1}{2}+\epsilon}\right)
		\end{equation*}
	\end{corollary}
	We remark that the iterated logarithms in Corollary \ref{gaposhkinCorr} can be extended arbitrarily far and we restrict to two iterations to simplify notation. We will apply this result to the process $\widehat{\mathbbm{1}}_{E_{T,c}}\circ g_t$, where 
	\begin{equation*}
		g_t = (\text{diag}(e^{a_{1\nu}t},\dots,e^{a_{1\nu}t},e^{-b_{m\nu}t}))_{\nu\in S}
	\end{equation*}
	By Moore's Ergodicity Theorem, this flow is ergodic on $K_S^d$ with respect to $\mu$. 
	
	The main difficulty in applying Theorem \ref{Rogers3} to the approximation problem (\ref{Dioph}) is that the error term involves summation over all elements in the number field $K$. To calculate this term, we first approximate the measure of the set $\lambda(E_{T,c}\cap \gamma^{-1}E_{T,c})$ for $\gamma \in K^\times$, then decompose the sum over $K^\times$ to a sum over all units, principal integral ideals $(q)$ and finally principal integral ideals $(p)$ so that $(\tfrac{p}{q})$ is principal. Such a decomposition allows us to make use of the Dirichlet Unit Theorem and the Dedekind--Weber Theorem to approximate the sum. This technique was used by Schanuel \cite{Schanuel1979} to effectively count the number of points in projective space of bounded height relative to $K$ and in a similar context by Kim \cite{kimAdelicRogersIntegral2022} for volumes of spheres and annuli in $K_S^d$.
	
	\subsection{Counting Units}
	First, we will prove a small combinatorial result.
	\begin{lemma}\label{Zcount}
		Let $n\geq 1$, $M,N \geq 0$ be integers, and
		\begin{equation*}
			\mathcal{Z}_{M,N}=\left\{ (z_i)\in \Z^n\,:\, \sum_{i}\max(0,z_i)=M, \, \sum_{i}\min(0,z_i)=-N \right\} 
		\end{equation*}
		Then $\#\mathcal{Z}_{M,N} = O\left(((M+N)^n\right)$.
	\end{lemma}
	\begin{proof}
		Clearly $\#\mathcal{Z}_{M,N}=0$ for $M=N=0$, so suppose at least one of these quantities is non-zero. For now, we will count elements $(z_i) = (x_1,\dots,x_k,0,\dots,y_1,\dots,y_l)$ with $z_i\leq z_{i+1}$ and $k+l \leq n$ such that $(x_i)$ is a partition of $M$ and $(y_i)$ is a partition of $N$. Let $p(k,M)$ equal the number of integer partitions of $M$ into $k$ integers, up to ordering. It was shown by Knessl and Keller in \cite{KnesslKeller} that
		\begin{equation*}
			p(k,M) \sim \frac{1}{2\pi M} e^{2k} \left(\frac{M}{k^2}\right)^k = O\left(M^{k-1}\right)
		\end{equation*}
		If $N=0$, then there exists $C>0$ such that
		\begin{gather*}
			\#\mathcal{Z}_{M,0} = \sum_{k=1}^n \tfrac{(\#S-1)!}{(\#S-1-k)!}p_k(M) 
			\leq C\sum_{k=1}^n \tfrac{(\#S-1)!}{(\#S-1-k)!}M^{k-1} \left(\tfrac{e}{k}\right)^{2k} 
			\leq C M^{n-1}
		\end{gather*}
		Similarly, if $M=0$, then $\#\mathcal{Z}_{0,N}= O(N^{n-1})$. If $M=N\neq 0$, then
		\begin{gather*}
			\#\mathcal{Z}_{M,M} = \sum_{k+l=2}^n p_k(M)p_l(M) 
			=\sum_{k=1}^{n-1} p_k(M) \sum_{l=1}^{n-k} p_l(M) \\
			\leq C \sum_{k=1}^{n-1} M^{k-1} M^{n-k} 
			\leq C nM^{n-1}
		\end{gather*}
		If $M\neq N$, both non-zero, then
		\begin{gather*}
			\#\mathcal{Z}_{M,N} = \sum_{k+l=2}^n p_k(M)p_l(N) 
			= \sum_{k=1}^{n-1} \left(p_k(M)\sum_{l=1}^{n-k} p_l(N) \right) \\
			\leq C \sum_{k=1}^{n-1} M^{k-1} N^{n-k} 
			\leq C (M+N)^{n}
		\end{gather*}
		
		Finally, in each case there are at most $n!$ permutations of each vector contained in $\mathcal{Z}_{M,N}$, so $\#\mathcal{Z}_{M,N}\leq Cn!(M+N)^n = O_n\left((M+N)^n\right)$.
	\end{proof}
	Define the sets
	\begin{equation*}
		F_{T,c} = \left\{ (\vb{x},\vb{y})\in K_S^m \times K_S^n\,:\, \norm{\vb{x}}_\mathbf{a}\norm{\vb{y}}_{\mathbf{b}}<c,\, \norm{\vb{y}}_\mathbf{b}<e^T \right\}
	\end{equation*}
	\begin{equation*}
		F_{T,c}^\nu = \left\{ (\vb{x}_\nu,\vb{y}_\nu)\in K_\nu^m \times K_\nu^n\,:\, \norm{\vb{x}_\nu}_{\mathbf{a}_\nu}\norm{\vb{y}_\nu}_{\mathbf{b}_\nu}<c,\, \norm{\vb{y}_\nu}_{\mathbf{b}_\nu}<e^T \right\}
	\end{equation*}
	where $\norm{\vb{x}_\nu}_{\mathbf{a}_\nu} = \max_{1\leq i \leq m}|x_{i\nu}|^{\frac{1}{a_{i\nu}}}$ and $\norm{\vb{y}_\nu}_{\mathbf{b}_\nu} = \max_{1\leq j \leq n}|y_{j\nu}|^{\frac{1}{b_{j\nu}}}$. Using elementary set operations, we see that
	\begin{align*}
		E_{T,c}\cap \alpha E_{T,c} 
		&= \left(F_{T,c}\setminus F_{0,c}\right) \cap \left(\alpha F_{T,c}\setminus \alpha F_{0,c}\right) \\
		&= \left(\prod_{\nu\in S} F_{T,c}^\nu \setminus \prod_{\nu\in S} F_{0,c}^\nu\right) \cap\left( \prod_{\nu\in S} \alpha F_{0,c}^\nu \setminus \prod_{\nu\in S} \alpha F_{0,c}^\nu\right) \\
		&=\prod_{\nu\in S}\left( F_{T,c}^\nu \cap \iota_\nu(\alpha) F_{T,c}^\nu \right)\setminus \prod_{\nu\in S} \left(F_{0,c}^\nu \cup \iota_\nu(\alpha) F_{0,c}^\nu\right) \\
		&= \prod_{\nu\in S} \min(1,|\iota_\nu(\alpha)|)F_{T,c}^\nu \setminus \prod_{\nu\in S} \max(1,|\iota_\nu(\alpha)|)F_{0,c}^\nu \\
		&\subseteq \prod_{\nu\in S} m_\nu(\alpha)F_{T,c}^\nu \setminus \prod_{\nu\in S}m_\nu(\alpha)F_{0,c}^\nu \\
		&= (m_\nu(\alpha))_{\nu\in S} \cdot \left(\prod_{\nu\in S} F_{T,c}^\nu \setminus \prod_{\nu\in S}F_{0,c}^\nu \right) \\
		&= (m_\nu(\alpha))_{\nu\in S} \cdot E_{T,c}
	\end{align*}	
	where $m_\nu(\alpha) = \min(1,|\iota_\nu(\alpha)|)$ and $(m_\nu(\alpha))_{\nu\in S}$ acts on $K_S^d$ via $(m_\nu(\alpha))\cdot (z_{i\nu})_{i,\nu} = (m_\nu(\alpha)z_{i_\nu})_{i\nu}$. Note that we can interpret $(m_\nu(\alpha))_{\nu\in S}$ as a linear transformation on $K_S^d$ with determinant $\prod_{\nu\in S}|m_\nu(\alpha)|^{d_\nu}$. 
	
	Using this and Lemma \ref{Zcount}, we can now give a bound for a sum over $\mathcal{O}_K^\times$, which will be used in the proof of theorem \ref{maintheorem}.
	\begin{lemma}\label{UnitBound}
		For any $\gamma \in K^\times$, there exists a constant $C>0$ such that
		\begin{equation*}
			\sum_{u\in \mathcal{O}_K^\times}\int_{K_S^d}\mathbbm{1}_{E_{T,c}}(\vb{z})\mathbbm{1}_{E_{T,c}}(u\gamma\cdot \vb{z})\,d\lambda(\vb{z})>0
		\end{equation*}
		is less than or equal to $C$.
	\end{lemma}
	\begin{proof}
		The units $\mathcal{O}_K^\times$ map to $\mathcal{L}=\{(z_\nu)\in \R^{\#S}\,:\,\sum_{\nu\in S}z_\nu=0\}$ via the function
		\begin{equation*}
			\text{Log}(u)=\left(d_\nu\log|\iota_\nu(u)|\right)_{\nu\in S}
		\end{equation*}
		By Dirichlet's Unit Theorem, the image of $\mathcal{O}_K^\times / \xi_K$ under $\text{Log}$ is a full rank lattice in $\mathcal{L}$, where $\xi_K$ is the maximal torsion group of $\mathcal{O}_K$. Partition $\R^{\#S}$ with unit $\#S$-cubes $Q(v_1,\dots,v_{\#S})$, indexed by their maximal vertex (for example, $[0,1]^{\#S} = Q(1,\dots,1)$). For each $u \in \mathcal{O}_K^\times$, $\text{Log}(\gamma u) \in Q(v_1,\dots,v_{\#S})$ for some $v_i \in \Z$, and
		\begin{gather*}
			\lambda\left(E_{T,c}\cap (\gamma u)^{-1}E_{T,c}\right) \leq \prod_{\nu\in S}|m_\nu(u\gamma)| \lambda(E_{T,c}) \\
			\leq \lambda(E_{T,c}) \prod_{\nu\in S} \min(1,e^{d_\nu v_\nu})
		\end{gather*}
		If $Q(v_1,\dots,v_{\#S})\cap \mathcal{L}\neq \emptyset$ then there exists a point $(y_\nu)_{\nu\in S}\in Q$ such that $\sum_{\nu\in S}y_\nu + \log N\gamma =0$, hence $\sum_{\nu\in S}v_\nu \geq \sum_{\nu\in S}y_\nu \geq -\log N\gamma$. Additionally, $Q(v_1-1,\dots,v_{\#S}-1)\cap \mathcal{L}=\emptyset$, so $\sum_{\nu\in S}v_\nu \leq \#S - \log N\gamma$. Let 
		\begin{equation*}
			\mathcal{U}_\gamma = \left\{ (v_\nu)_{\nu\in S}\in \Z^{\#S} \,:\, - \log N \gamma \leq \sum v_\nu \leq \#S-\log N \gamma \right\}
		\end{equation*}
		By definition, the set $\mathcal{U}_\gamma$ contains all unit cubes in $\R^{\#S}$ that intersect with the hyperplane containing $\text{Log}(\gamma\mathcal{O}_K^\times)$. Ball's Cube Slicing Theorem states that $\lambda_{\#S-1}(\mathcal{L}\cap Q)\leq \sqrt{2}$ for any cube $Q((v_\nu)_{\nu\in S})$ with $(v_\nu)_{\nu\in S} \in \mathcal{U}_\gamma$ \cite{ballCubeSlicingInRn}, so using a lattice counting argument, for example \cite[Theorem 5.4]{widmerCountingPrimitivePoints2010}, we can bound
		\begin{gather*}
			\#(\text{Log}(\mathcal{O}_K^\times)\cap Q) \leq \frac{\sqrt{2}}{R_K} + O_K(1)
		\end{gather*}
		where $R_K$ is the regulator of $K$. Then
		\begin{gather*}
			\sum_{u \in \mathcal{O}_K^\times} \int_{K_S^d}\mathbbm{1}_{E_{T,c}}(\vb{z})\mathbbm{1}_{E_{T,c}}(u\gamma\cdot \vb{z})\,d\lambda(\vb{z}) \\
			\leq C_K\lambda(E_{T,c}) \sum_{(v_\nu)\in \mathcal{U}_\gamma} \prod_{\nu\in S} \min(1,e^{d_\nu v_\nu})
		\end{gather*}
		for some constant $C_K>0$. Fix $\sigma \in S$, let $S_\sigma = S\setminus \{\sigma\}$ and decompose
		\begin{gather*}
			\Z^{\#S-1} = \bigsqcup_{M,N\in \Z^2} \mathcal{Z}_{M,N} \\:= \bigsqcup_{M,N\in \Z^2} \left\{(v_\nu)_{\nu\in S_\sigma}\in \Z^{\#S-1}\,:\, \sum_{\nu\in S_\sigma} \min(0, v_\nu)=-M, \sum_{\nu\in S_\sigma}\max(0, v_\nu)=N\right\}
		\end{gather*}
		Write $f(x) \lesssim g(x)$ if $f(x) = O(g(x))$ and set $w=\#S-\log N\gamma$. We have
		\begin{gather*}
			\sum_{(v_\nu)\in \mathcal{U}_\gamma} \prod_{\nu\in S} \min(1,e^{v_\nu}) \nonumber \\
			\leq \sum_{(v_\nu)\in S_\sigma} \sum_{v_\sigma \,:\, (v_\sigma,v_\nu)\in\mathcal{U}_\gamma} \prod_{\nu\in S} \min(1,e^{ v_\nu})\nonumber \\
			= \sum_{M=0}^\infty \sum_{N=0}^\infty \sum_{(v_\nu)\in \mathcal{Z}_{M,N}} \sum_{v_\sigma\,:\, (v_\sigma,v_\nu)\in \mathcal{U}_\gamma}  e^{-M} \min(1,e^{ v_\sigma})\nonumber \\
			\lesssim \sum_{M=0}^\infty \sum_{N=0}^\infty \sum_{(v_\nu)\in \mathcal{Z}_{M,N}}  e^{-M} \min(1,e^{w + M-N})\nonumber 
		\end{gather*}
		By Lemma \ref{Zcount}, $\#\mathcal{Z}_{M,N} = O\left((M+N)^{\#S-1}\right)$. Therefore
		\begin{gather} 
			\sum_{M=0}^\infty \sum_{N=0}^\infty \sum_{(v_\nu)\in \mathcal{Z}_{M,N}}  e^{-M} \min(1,e^{w + M-N})\nonumber\\
			\lesssim  \sum_{M=0}^\infty \sum_{N=0}^\infty (M+N)^{\#S-1}  \min(e^{-M},e^{w -N}\nonumber) \\
			=  \sum_{M=0}^\infty \left(\sum_{N=0}^\infty \sum_{k=1}^{\#S-1} \begin{pmatrix}\#S-1\nonumber \\ k\end{pmatrix} M^{\#S-1-k}N^k \min(e^{-M},e^{w -N}) \right)\nonumber \\
			\begin{gathered}\label{sum}
			=    \sum_{k=1}^{\#S-1}\begin{pmatrix}\#S-1 \\ k\end{pmatrix} M^{\#S-1-k} \left(\sum_{M=0}^\infty  e^{-M}\sum_{N=0}^{\lfloor w  \rfloor+ M }N^k\right) \\
			+ \sum_{k=1}^{\#S-1}\begin{pmatrix}\#S-1 \\ k\end{pmatrix}\left( \sum_{M=0}^\infty  M^{\#S-1-k}\sum_{N=\lceil  w \rceil+ M}^\infty N^k e^{w - N}  \right)
			\end{gathered}
		\end{gather}
		The first summand can be estimated as
		\begin{equation}\label{asym1}
		\begin{gathered}
			\sum_{M=0}^\infty M^{\#S-1-k}e^{-M}\sum_{N=0}^{\lfloor w\rfloor+M} N^k \leq \sum_{M=0}^\infty M^{\#S-1-k}e^{-M}(w +M)^{k+1} \\
			\lesssim \sum_{M=0}^\infty M^{\#S}e^{-M}
		\end{gathered}
	\end{equation}
		Similarly, the second summand is bounded by
		\begin{equation}\label{asym2}
		\begin{gathered}
			\sum_{M=0}^\infty\sum_{N=\lfloor w \rfloor+M}^\infty N^k e^{w -N} 
			\leq \sum_{M=0}^\infty (M+1)(w + M)^k e^{-M} \\
			\lesssim \sum_{M=0}^\infty M^{k+1}e^{-M}
		\end{gathered}
		\end{equation}
		Combining expressions (\ref{sum}), (\ref{asym1}), and (\ref{asym2}), we get
		\begin{gather*}
			\sum_{(v_\nu)\in \mathcal{U}_\gamma} \prod_{\nu\in S} \min(1,e^{v_\nu}) \\
			\lesssim \sum_{M=0}^\infty \sum_{k=1}^{\#S-1} \begin{pmatrix} \#S-1\\ k\end{pmatrix}e^{-M}(M^{k+1}+M^{\#S}) \\
			\lesssim  \sum_{M=0}^\infty e^{-M} M^{\#S} < \infty 
		\end{gather*}
		Since this sum is only dependent on $K$, the claim is proven.
	\end{proof}
	
	\subsection{Proof of Theorems \ref{maintheorem} and \ref{spiralling}}
	
	Before we prove the main theorems, we recall a result that counts the density of integral ideals with bounded norms in $\mathcal{O}_K$.
	\begin{theorem}[{{\cite[Theorem 121]{heckeLecturesTheoryAlgebraic1981}}}]\label{dedekindweber}
		Let $K$ be a number field. The number of principal ideals $I\lhd\mathcal{O}_K$ of norm less than or equal to $s$ is given by
		\begin{equation*}
			s \chi_K  + O(s^{1-\frac{1}{\deg K}})
		\end{equation*}
		where $\chi_K>0$ is constant.
	\end{theorem}
	This theorem will be used in tandem with the following lemma by Thunder to calculate a sufficient upper bound for some of the terms appearing in Theorem \ref{Rogers3}.
	\begin{lemma}[{{\cite[Lemma 12]{Thunder1992AnAE}}}]\label{thunder}
		Let $\mathcal{G}$ and $g:\mathcal{G}\rightarrow [1,\infty)$. Suppose 
		\begin{equation*}
			\textnormal{card}\{g^{-1}[1,x]\} = c_1x^{c_2}+O(x^{c_3})
		\end{equation*} 
		where $c_1,c_2,c_3>0$. Let $t\geq 1$ and $F\in C^1((0,t);\R)$ be such that $F'(x)\leq 0$ and $F'$ decreasing on $(0,t)$. Then
		\begin{equation*}
			\sum_{w \in \mathcal{G}} F(g(w))<\int_{\frac{1}{2}}^t x^{c_2-1}F(x)\,dx + 2^{-c_2}F\left(\tfrac{1}{2}\right)
		\end{equation*}
	\end{lemma}
	
	\begin{proof}[proof of Theorem \ref{maintheorem}]
		As shown before $\mathcal{D}_{1,2} = \left\{ (1,\gamma)\,:\,\gamma\in K^\times \right\}$, hence by Theorem \ref{Rogers3}, 
		\begin{gather*}
			\int_X \widehat{\mathbbm{1}}_{E_{T,c}}^2(\Lambda)d\mu(\Lambda) = \left(\int_{K_S^d} \mathbbm{1}_{E_{T,c}}(\vb{x})\,d\lambda(\vb{x})\right)^2 \\
			+  \sum_{c \in K^\times} \frac{1}{\text{covol}(\Phi((1,\gamma)))} \int_{K_S^d} \mathbbm{1}_{E_{T,c}}(\vb{x})\mathbbm{1}_{E_{T,c}}(\gamma\cdot \vb{x})\,d\lambda(\vb{x}) \label{eq1}
		\end{gather*}
		For $q \in \mathcal{O}_K$, let $\mathcal{E}_q$ be the set of all $(p) \lhd \mathcal{O}_K$ coprime to $(q)$ such that $(\gamma)=\left(\tfrac{p}{q}\right)$ is a proper principal ideal of norm less than $1$. Implicitly, we choose representatives $p$ of each ideal $(p)\in \mathcal{E}_q$ that maximise the integral above. Furthermore, since
		\begin{equation*}
			\int_{K_S^d} \mathbbm{1}_{E_{T,c}}(\vb{x})\mathbbm{1}_{E_{T,c}}(\gamma^{-1}\cdot \vb{x})\,d\lambda(\vb{x}) = N(\gamma)^{d} \int_{K_S^d} \mathbbm{1}_{E_{T,c}}(\vb{x})\mathbbm{1}_{E_{T,c}}(\gamma\cdot \vb{x})\,d\lambda(\vb{x}) 
		\end{equation*}
		for $\gamma=\tfrac{p}{q}$ with $N(\gamma)\leq1$, we can calculate
		\begin{align*}
			&\frac{1}{\text{covol}(\Phi(1,\gamma))}\int_{K_S^d} \mathbbm{1}_{E_{T,c}}(\vb{x})\mathbbm{1}_{E_{T,c}}(\gamma\cdot \vb{x})\,d\lambda(\vb{x}) \\
			&+ \frac{1}{\text{covol}(\Phi(1,\gamma^{-1}))}\int_{K_S^d} \mathbbm{1}_{E_{T,c}}(\vb{x})\mathbbm{1}_{E_{T,c}}(\gamma^{-1}\cdot \vb{x})\,d\lambda(\vb{x}) \\
			=&\frac{1}{N(q)^d}\int_{K_S^d} \mathbbm{1}_{E_{T,c}}(\vb{x})\mathbbm{1}_{E_{T,c}}(\tfrac{p}{q}\cdot \vb{x})\,d\lambda(\vb{x}) \\
			&+ \frac{1}{N(p)^d}\int_{K_S^d} \mathbbm{1}_{E_{T,c}}(\vb{x})\mathbbm{1}_{E_{T,c}}(\tfrac{q}{p}\cdot \vb{x})\,d\lambda(\vb{x}) \\
			=& \frac{2}{N(q)^{d}} \int_{K_S^d} \mathbbm{1}_{E_{T,c}}(\vb{x})\mathbbm{1}_{E_{T,c}}(\tfrac{p}{q}\cdot \vb{x})\,d\lambda(\vb{x}) 
		\end{align*}
		Using the above, Lemma \ref{UnitBound} and the Cauchy-Schwartz inequality, we can then bound (\ref{eq1}) by
		\begin{gather*}
			\sum_{(q)\lhd\mathcal{O}_K}\sum_{(p)\in \mathcal{E}_q} \sum_{u\in\mathcal{O}_K^\times} \frac{1}{N(q)^{d}}\int_{K_S^d} \mathbbm{1}_{E_{T,c}} (\vb{x})\mathbbm{1}_{E_{T,c}}\left(\tfrac{up}{q}\cdot\vb{x}\right)\,d\lambda(\vb{x}) \\
			\lesssim \lambda(E_{T,c})+\sum_{(q)\lhd\mathcal{O}_K}\sum_{{(p)\in \mathcal{E}_q}}\frac{1}{N(q)^{d}} \int_{K_S^d} \mathbbm{1}_{E_{T,c}} (\vb{x})\mathbbm{1}_{E_{T,c}}\left(\tfrac{p}{q}\cdot\vb{x}\right)\,d\lambda(\vb{x}) \\
			\leq \lambda(E_{T,c})+\lambda(E_{T,c})^\frac{1}{2}\sum_{(q)\lhd\mathcal{O}_K}\sum_{{(p)\in \mathcal{E}_q}}\frac{1}{N(q)^{d}} \left(\int_{K_S^d}\mathbbm{1}_{E_{T,c}}\left(\tfrac{p}{q}\cdot\vb{x}\right)\,d\lambda(\vb{x})\right)^\frac{1}{2} \\
			= \lambda(E_{T,c})\left(1+ \sum_{(q)\lhd\mathcal{O}_K}\sum_{\stackrel{(p)\in \mathcal{E}_q}{N(p)<N(q)}}\frac{1}{N(p)^{\frac{d}{2}}N(q)^{\frac{d}{2}}} \right)
		\end{gather*}
		By Theorem \ref{dedekindweber} and Lemma \ref{thunder}, 
		\begin{gather*}
			\lambda(E_{T,c})\left(1+ \sum_{(q)\lhd\mathcal{O}_K}\sum_{\stackrel{(p)\in \mathcal{E}_q}{N(p)<N(q)}}\frac{1}{N(p)^{\frac{d}{2}}N(q)^{\frac{d}{2}}}\right)  \\
			< \lambda(E_{T,c})\left(1+ \sum_{(q)\lhd \mathcal{O}_K} \left(\frac{1}{(1-\frac{d}{2})N(q)^{d-1}} + \frac{2^{\frac{d}{2}-1}}{N(q)^{\frac{d}{2}}(\frac{d}{2}-1)} + \frac{1}{N(q)^\frac{d}{2}}\right) \right)\\
			\leq \lambda(E_{T,c})\left(1+\frac{1}{1-\frac{d}{2}}\zeta_K(d-1) + \frac{2^{\frac{d}{2}-1}}{\frac{d}{2}-1}\zeta_K(\tfrac{d}{2}) + \zeta_K(\tfrac{d}{2}) \right) \\
			=O(\lambda(E_{T,c}))
		\end{gather*}
		On the other hand, $\lambda(E_{T,c})$ is clearly a lower bound for this expression, hence the second moment of the function $\widehat{\mathbbm{1}}_{E_{T,c}}$ has the asymptotic bound
		\begin{equation}\label{asbound}
			\left|\int_X \left(\widehat{\mathbbm{1}}_{E_{T,c}} -\lambda(E_{T,c})\right)^2\,d\mu\right|\sim O(\text{vol}(E_{T,c}))
		\end{equation}
		
		We can now show that $\widehat{\mathbbm{1}}_{E_{T,c}}\circ g_t$ is a mean-square continuous wide-sense stationary stochastic process. By Theorem \ref{Rogers3}, for $s\leq t \in\R$, we can write
			\begin{gather*}
				\int_X \widehat{\mathbbm{1}}_{E_{T,c}}(g_s\Lambda) \widehat{\mathbbm{1}}_{E_{T,c}}(g_t\Lambda) \,d\mu(\Lambda) = \lambda(g_{-s}E_{T,c})\lambda({g_{-t}}E_{T,c}) \\
				+ \sum_{c \in K^\times} \int_{K_S^d} \mathbbm{1}_{E_{T,c}}(g_s \vb{x})\mathbbm{1}_{E_{T,c}}(cg_t\vb{x})\,d\lambda(\vb{x}) \\
				= \lambda(E_{T,c})^2 + \sum_{c\in K^\times} \lambda(g_{-s}E_{T,c}\cap c^{-1}g_{-t}E_{T,c})
			\end{gather*}
		Since the flow $g_t$ is continuous, the expression inside the sum is continuous in $s$ and $t$ for all $c\in K^\times$. Applying the measure-preserving transformation $g_{\tfrac{t-s}{2}}$ then shows the process is mean-square continuous. The $g_t$-invariance of the measure $\lambda$ also implies that, for any $\tau\in \R$, 
		\begin{gather*}
			\lambda(g_{-s}E_{T,c}\cap c^{-1}g_{-t}E_{T,c}) = \lambda\left(g_{\tau}\left(g_{-s}E_{T,c}\cap c^{-1}g_{-t}E_{T,c}\right)\right)\\=\lambda(g_{\tau-s}E_{T,c}\cap c^{-1}g_{\tau-t}E_{T,c})
		\end{gather*}
		showing the process is wide-sense stationary.
		
		Equation (3.2) of \cite{alamEquidistributionHomogeneousSpaces2020} states that, for $1<R<T$,
		\begin{gather*}
			\#((E_{T,c}\setminus E_{R,c})\cap \Lambda) \leq \frac{1}{R} \int_0^T \widehat{\mathbbm{1}}_{E_{R,c}}(g_t\Lambda)\,dt \leq \#(E_{R+T,c}\cap \Lambda)
		\end{gather*}
		Squaring both sides, integrating over $X$ with respect to $\Lambda$, we get
		\begin{equation*}
			\int_X \widehat{\mathbbm{1}}_{E_{T,c}\setminus E_{R,c}}^2\,d\mu \leq \frac{1}{R^2}\int_X \left(\int_0^T \widehat{\mathbbm{1}}_{E_{R,c}}(g_t\Lambda)\,dt\right)^2\,d\mu \leq \int_X \widehat{\mathbbm{1}}_{E_{T+R,c}}\,d\mu 
		\end{equation*}
		It is easy to see that $\lambda(E_{T,c})=\alpha T$ for some $\alpha >0$ depending on $K$, $c$ and $d$ only. Equation (\ref{asbound}) allows us to express this as, for some $c_1,c_2 >0$,
		\begin{align*}
			&\alpha^2(T-R)^2 + \alpha^2c_1(T-R) \\\leq &\frac{1}{T^2}\int_X \left(\int_0^R \widehat{\mathbbm{1}}_{E_{R,c}}(g_t\Lambda)\,dt\right)^2\,d\mu \\\leq &\alpha^2(T+R)^2 + \alpha c_2(T+R)
		\end{align*}
		With some rearranging, we see that
		\begin{equation*}
			\left| \int_X\left(\frac{1}{T}\int_0^T \widehat{\mathbbm{1}}_{E_{R,c}}^2(g_t\Lambda)\,dt - \lambda(E_{R,c}) \right)^2\,d\mu\right| =O_R\left(\frac{1}{T}\right) 
		\end{equation*}
		and Corollary \ref{gaposhkinCorr} then implies the result.
		
		Furthermore, for $\nu$-almost every $\vartheta \in \text{Mat}_{m\times n }(K_S)$, the lattice $\Lambda_\vartheta$ is Birkhoff-generic with respect to the function $\widehat{\mathbbm{1}}_{E_{T,c}}$ \cite[Theorem 2.3]{alamEquidistributionHomogeneousSpaces2020}, so we may also conclude 
		\begin{equation*}
			\#(\Lambda_\vartheta \cap E_{T,c}) = \lambda(E_{T,c}) + o\left(T^\frac{1}{2}\left(\log T\right)^{\frac{3}{2}}\left(\log\log T\right)^{\frac{1}{2}+\epsilon}\right) 
		\end{equation*}
		for $\nu$-almost every $\vartheta \in \text{Mat}_{m\times n}(K_S)$.
	\end{proof} 
	\begin{proof}[proof of Theorem \ref{spiralling}]
		The proof of this theorem follows closely the proof above. Since $\lambda(E_{T,c}(A,B)) = \lambda(E_{T,c})\text{vol}(A)\text{vol}(B)$, a similar calculation to that in the above proof shows
		\begin{equation*}
			\left|\int_X\left(\widehat{\mathbbm{1}}_{E_{T,c}(A,B)}-\lambda(E_{T,c}(A,B))\right)^2 \,d\mu \right| = O(\lambda(E_{T,c}(A,B)))
		\end{equation*}
		and by applying Corollary \ref{gaposhkinCorr} as above, we prove the theorem.
	\end{proof}
	
	\bibliographystyle{abbrv}
	\typeout{}
	\bibliography{./Bibliography}
	
\end{document}